\documentclass[12pt]{amsart}

\usepackage[utf8]{inputenc}
\usepackage{color}
\usepackage[centertags]{amsmath}
\usepackage[margin=1in]{geometry}
\usepackage{amsfonts}
\usepackage{amsthm}
\usepackage{caption}
\usepackage{newlfont}
\usepackage{verbatim}
\usepackage{amscd}
\usepackage{amsgen}
\usepackage{comment}
\usepackage{xspace}
\usepackage{amssymb}
\usepackage{stmaryrd}
\usepackage{amssymb,amsmath}
\usepackage{amscd}
\usepackage{enumitem}
\usepackage{color}
\usepackage{xy}
\usepackage{graphicx}
\usepackage{tikz}
\usetikzlibrary{arrows,shapes,trees}
\usetikzlibrary{backgrounds}
\usetikzlibrary{positioning}

\newlength{\defbaselineskip} 
\theoremstyle{plain}
\newtheorem{thm}{Theorem}
\newtheorem{cor}[thm]{Corollary}
\newtheorem{con}[thm]{Conjecture}
\newtheorem{lema}[thm]{Lemma}
\newtheorem{prop}[thm]{Proposition}
\newtheorem{obs}[thm]{Observation}

\newtheorem{ques}[thm]{Question}
\theoremstyle{definition}
\newtheorem{df}[thm]{Definition}
\newtheorem{exm}[thm]{Example}
\newtheorem{rem}[thm]{Remark}

\newtheorem{fact}[thm]{Fact}

\newtheorem{pr}{Algorithm}

\def\Z{\mathbb{Z}}
\def\Q{\mathbb{Q}}

\def\ob{\begin{obs}}
\def\kob{\end{obs}}
\def\dow{\begin{proof}}
\def\kdow{\end{proof}}

\def\tw{\begin{thm}}
\def\ktw{\end{thm}}
\def\hip{\begin{con}}
\def\khip{\end{con}}
\def\lem{\begin{lema}}
\def\klem{\end{lema}}
\def\ex{\begin{exm}}
\def\prog{\begin{pr}}
\def\kprog{\end{pr}}
\def\wn{\begin{cor}}
\def\kwn{\end{cor}}
\def\uwa{\begin{rem}}
\def\kuwa{\end{rem}}
\def\kex{\end{exm}}
\def\dfi{\begin{df}}
\def\kdfi{\end{df}}

\definecolor{zielony}{rgb}{0.5, 0.9, 0.1}
\definecolor{czerwony}{rgb}{0.9, 0.2, 0.1}
\definecolor{niebieski}{rgb}{0.3, 0.1, 0.9}

\xyoption{line}
\def\fa{\begin{fact}}
\def\kfa{\end{fact}}

\usepackage{url}
\usepackage{hyperref}
\definecolor{darkblue}{RGB}{0,0,160}
\hypersetup{
colorlinks,%
citecolor=black,%
filecolor=black,%
linkcolor=darkblue,%
urlcolor=darkblue
}

\title{Obstructions to combinatorial formulas for plethysm}

\author{Thomas Kahle} \address{Fakult\"at f\"ur Mathematik,
Otto-von-Guericke Universit\"at, 39106 Magdeburg, Germany}
\urladdr{\url{http://www.thomas-kahle.de}}

\author{Mateusz Micha\l ek} \address{Max Planck Institut for
Mathematics in the Sciences, Leipzig, Germany and Mathematical
Institute of the Polish Academy of Sciences, Warsaw, Poland}
\email{wajcha2@poczta.onet.pl}

\makeatletter
  \@namedef{subjclassname@2010}{\textup{2010} Mathematics Subject Classification}
\makeatother

\date{}

\begin{document}
\begin{abstract}
Motivated by questions of Mulmuley and Stanley we investigate
quasi-poly\-no\-mials arising in formulas for plethysm.  We
demonstrate, on the examples of $S^3(S^k)$ and $S^k(S^3)$, that these
need not be counting functions of inhomogeneous polytopes of dimension
equal to the degree of the quasi-polynomial. It follows that these
functions are not, in general, counting functions of lattice points in
any scaled convex bodies, even when restricted to single rays. Our
results also apply to special rectangular Kronecker coefficients.
\end{abstract}

\maketitle

Many problems in representation theory have combinatorial solutions.
A~well-known, important example is the Littlewood--Richardson rule
which gives the multiplicities of isotypic components in the tensor
product of two irreducible $GL(n)$ representations
\cite{fulton91:_repres_theor}.  The solution is combinatorial in the
sense that the answer is given by the number of lattice points in
explicit rational convex polyhedra, i.e.~the multiplicities are equal
to values of an Ehrhart quasi-polynomial.  Berenstein and Zelevinsky
provided another interpretation of the Littlewood--Richardson
rule~\cite{BZ}, isomorphic, however, on the level of polytopes
\cite{pak2005combinatorics}.  The study of different polyhedral
structures turned out to be very useful.  Knutson and Tao
\cite{knutson1999honeycomb} showed that the honeycomb polytopes (in
spirit similar to \cite{BZ}) are nonempty if and only if they contain
a lattice point.  This was the crucial last step in the solution to
the Horn problem~\cite{horn1962eigenvalues} which goes back to Weyl's
work on eigenvalues of partial differential
equations~\cite{weyl1912asymptotische}.

In this paper we study the much more complicated plethysm coefficients
$m^{d,k}_\lambda$ defined by
\[
S^d(S^k W)=\bigoplus_{\lambda} m^{d,k}_\lambda S^\lambda W,
\]
where the sum is over partitions $\lambda$ of~$dk$.

The computation of the multiplicity functions in plethysms can be seen
as an operation on Schur polynomials~\cite{macdonald1998symmetric}.
Viewed like this, it is surprising that when the plethysm is written
in terms of other Schur polynomials, the coefficients are always
nonnegative.  Of course this must be true, because the coefficients
are multiplicities of irreducible representations, but it would be
desirable to have a combinatorial explanation for nonnegativity.
In \cite[Problem~9]{Stanley} Richard Stanley asked for a
\emph{positive} combinatorial method to compute plethysm coefficients.
A connection between plethysm and lattice point counting was shown at
least
in~\cite{KMplethysm,colmenarejo2015stability,christandl2012computing}.
These connections are not direct in the sense that plethysm
coefficients are not seen to equal counts, but always involve some
opaque arithmetics.

The functions $f(s)=m^{d,sk}_{s\lambda}$ and
$g(s)=m^{sd,k}_{s\lambda}$ share many properties with Ehrhart
functions of rational polytopes:
\begin{itemize}
\item Both $f$ and $g$ are nonnegative and $f(0)=g(0)=1$.
\item Both $f$ and $g$ are quasi-polynomials.  This deep fact follows
from~\cite[Corollary~2.12]{meinrenken1999singular}.
See~\cite[Remark~3.12]{KMplethysm}.
\item For regular $\lambda$, the leading term of $f$ is proportional
to the Ehrhart function of a rational polytope corresponding to the
Littlewood--Richardson rule~\cite[Section~4]{KMplethysm}.
\end{itemize}
\begin{rem}
The scaling factor between the leading coefficient of $f$ and the
Ehrhart function in the third bullet is the Plancherel measure on
Young diagrams.  It is conjectured that the same is true for
any~$\lambda$, however we are not aware of a proof.
\end{rem}

\begin{exm}
The multiplicity of $\lambda=s\cdot(2k-1,1)$ inside $S^2(S^{sk})$
equals $0$ when $s$ is odd and $1$ when $s$ even.  In particular, it
is equal to the number of integral points in the (zero-dimensional)
polytope $\{s/2\}$.  The multiplicity of $s\cdot(2k,k)$ inside
$S^3(S^{sk})$ is equal to the number of integral points in the
(one-dimensional) polytope $s\cdot[1/3,1/2]$ or $s\cdot [1/2,2/3]$.
\end{exm}

Ketan Mulmuley posed several conjectures concerning the behavior of
plethysm and Kronecker coefficients in relation to geometric
complexity theory \cite{mulmuley2001geometric,mulmuleyGCT6long}.
\begin{ques}\label{q:combinatorial}
Fix positive integers $k,d$, and $\lambda$ a partition of $kd$.  Are
the quasi-polynomials $s\mapsto m^{d,sk}_{s\lambda}$ and
$s\mapsto m^{sd,k}_{s\lambda}$ Ehrhart functions of rational
polytopes?
\end{ques}
The question about $m^{sd,k}_{s\lambda}$ comes directly from
\cite[Hypothesis~1.6.4]{mulmuleyGCT6long}.  We give a negative answer
for both functions in Remark~\ref{r:ehrhart-proof}.  Our main goal,
however, is a generalized version of these questions for which we need
some additional terminology.  Following
\cite[Section~5.1]{mulmuleyGCT6short} we define a \emph{shifted} or
\emph{inhomogeneous rational polytope} as a system of inequalities
\[
P (A,b,c) = \{x\in\mathbb{Q}^m : Ax \le b+c\},
\]
where $b,c$ are arbitrary rational vectors and $A$ is a rational
matrix.  Splitting the right hand side as $b + c$ is motivated by the
definition of the dilations of $P$ as
\[
P(A,sb,c) = \{x\in\mathbb{Q}^m : Ax \le sb + c\}.
\]
An \emph{asymptotic Ehrhart quasi-polynomial} is a counting function
of the form
\[
s\mapsto \#(P(A,sb,c)\cap\Z^m).
\]
The \emph{dimension} of a shifted rational polytope $P$ is by
definition the dimension of $P(A,sb,c)$ for large~$s$ (for small $s$ the
polytope can be empty).  Contrary to the case of Ehrhart
quasi-polynomials, an asymptotic Ehrhart quasi-polynomial does not
need to be a quasi-polynomial, although it is for large arguments.
Moreover, a quasi-polynomial may be an asymptotic Ehrhart
quasi-polynomial but not an Ehrhart quasi-polynomial.  Asymptotic
Ehrhart quasi-polynomials do not have to satisfy Ehrhart reciprocity.
Further, the dimension of a shifted rational polytope can be strictly
greater than the degree of the associated asymptotic Ehrhart
quasi-polynomial.  See~\cite{stanley1982linear}
and~\cite[Chapter~I]{stanley96:_combin_commut_algeb} for structure
theory of asymptotic Ehrhart quasi-polynomials and the relation to
linear diophantine equations.

In \cite[Hypothesis~5.3]{mulmuleyGCT6short} it is conjectured that the
multiplicity of $s\cdot\pi$ in $S^{s\cdot\lambda}(S^\mu)$ is an asymptotic
Ehrhart quasi-polynomial with additional complexity-theoretic
properties.  The asymptotic part of the conjecture is motivated by the
failure of Ehrhart type explanations already for Kronecker
coefficients~\cite{briand2009reduced,King} which are often considered
less complicated than plethysm coefficients.  In the present paper we
make progress towards a negative answer of the following simplified
version.
\begin{ques}\label{q:mulmuley2}
Is $m^{sd,k}_{s\lambda}$ an asymptotic Ehrhart quasi-polynomial?
\end{ques}

Our main result, Theorem~\ref{t:main}, implies a negative answer to
both cases in Question~\ref{q:combinatorial} and strong restrictions
on a positive solution to Question~\ref{q:mulmuley2}.  We show that
$m^{sd,k}_{s\lambda}$ need not be an asymptotic Ehrhart
quasi-polynomial of a polytope of dimension equal to the degree of its
growth.  At the moment we are not able either to exclude or confirm
that $m^{sd,k}_{s\lambda}$ is an asymptotic Ehrhart quasi-polynomial
for an inhomogeneous polytope of dimension strictly larger than its
degree.


In previous work the authors gave a formula for plethysm coefficients,
which is a sum of Ehrhart functions of various polytopes with
(positive and negative) coefficients~\cite{KMplethysm}.  This allows
to gather experimental data on the questions for many rays.
A-posteriori, the specific plethysm in Theorem~\ref{t:main} can also
be confirmed using well-known formulas for $S^3(S^k)$.  General
formulas for $S^k(S^3)$ are unknown, though.  Our methods are inspired
by \cite{King, briand2009reduced}.
\begin{thm}\label{t:main}
The multiplicity functions of $S^{s\cdot(7,5,0)}$ in $S^3(S^{4s})$ and
$S^{4s}(S^3)$ equal
\begin{align*}
  \phi: s & \mapsto \frac{s+r(s)}{3}
%
%
%
%
%
\end{align*}
where $r(s)$ has period $6$ and takes the values $3,-1,1,0,2,-2$ on
respectively the integers $0,\dots,5$.  There do not exist rational
vectors $a,b,c \in\mathbb{Q}^n$ (of arbitrary length $n$) such that
$\phi$ equals the counting function of a one-dimensional inhomogeneous
polytope $P(a,b,c)$.
%
In particular, $\phi$ is not an Ehrhart function of any rational
polyhedron.
\end{thm}
\begin{rem}\label{r:ehrhart-proof}
Before giving a general proof, it is easy to see that $\phi$ cannot
be an Ehrhart function: it violates Ehrhart--Macdonald
reciprocity~\cite{ehrhart1967I,macdonald1971polynomials}.  The value
$f(-s)$ of an Ehrhart quasi-polynomial $f$ of a rational polytope $P$
at a negative integer argument counts (up to a global sign) the number
of interior lattice points in~$s\cdot P$.  In particular,
$|f(-s)|\leq f(s)$.  However,
\[
-\phi(-1)=1>0=\phi(1).
\]
Interestingly, the jumps in values that make Ehrhart reciprocity fail
are also the crucial ingredient for our proof of non-representability
by a one-dimensional inhomogeneous polytope.
\end{rem}
\begin{proof}
The equality of the multiplicities of the given ray in both plethysms
is a consequence of Hermite reciprocity \cite{hermite1854},
\cite[Exercise~6.18]{fulton91:_repres_theor}.  There are now various
ways to determine the formula from the interpretation as a
multiplicity in $S^3(S^{4s})$.  One is to simply evaluate the explicit
formula from~\cite{KMplethysm} along a ray.  Another way is to observe
that the function must exhibit linear growth and that its period is at
most six~\cite{CGR83, manivel2014secants}.  With this information and
some values the function can be interpolated.

Let $P$ be an inhomogeneous line segment as in the statement.  Then
the $s$-th dilation of $P$ can be written as
\[
P (s) = \{x\in\Q: \max_i (sb_i + c_i) \le x \le \min_i (sb'_i + c'_i)\}.
\]
Asymptotically, $P(s)$ becomes an interval of length~$\frac{s}{3}$.
This means that there exists an $s_0$ such that for all $s > s_0$
\begin{equation}\label{eq:main_bound}
P(s) = \{x\in\Q: s b + c \le x \le s(b+\textstyle\frac{1}{3}) + c')\} =: Q(s),
\end{equation}
for some $b,c,c'\in\Q$.  Making $s_0$ even larger it can be assumed
that $s_0b$ are integers and that $s_0$ is divisible by~6 (the period
of~$\phi$).  Since $\#Q(s)$ and $\#P(s)$ have the same linear term and
constants that are $s_0$-periodic (by the divisibility assumptions)
they agree for all~$s$.

The proof of nonexistence of the family $Q(s)$ such that
$\phi(s) = \#Q(s)$ is by examination of constraints on $b$, coming
from the known values of the counting function~$\phi$.  Without loss
of generality we have
\begin{equation}
\label{eq:shiftP2} 0 \le b < 1.  
\end{equation}
Indeed, changing $b$ by an integer also only shifts $Q$ by an integral
value.

As $\phi(5)+2=\phi(6)$ there must be at least two integers $x$ such
that $5(b+\frac{1}{3})+c'<x\leq 6(b+\frac{1}{3})+c'$. Hence,
$b+\frac{1}{3}>1$.  In particular, for any $s$ there is always at
least one integer $x$ satisfying
$s(b+\frac{1}{3})+c'<x\leq (s+1)(b+\frac{1}{3})+c'$ (an interval of
length $>1$ must contain an integer). On the other hand, by
\eqref{eq:shiftP2} there is at most one integer $x$ satisfying
$sb+c'<x\leq (s+1)b+c'$ (an interval of length $<1$ may contain at
most one integer).  The above two intervals are the difference between
$P(s)$ and $P(s+1)$.  It follows that $\phi(s)$ is nondecreasing. This
contradicts $\phi(0)=1>0=\phi(1)$.
\end{proof}

\begin{rem}
We have confirmed that all examples with two rows and fewer boxes are
in fact Ehrhart functions.  In this sense, our counter-example is
minimal.
\end{rem}

\begin{rem}
A different way to describe an inhomogeneous polytope is by linear
integral equations and nonnegativity.  In this setting Stanley gave a
general reciprocity theorem which relates the Hilbert series (expanded
at $\infty$) of the module of positive integral solutions of linear
Diophantine equations to the Hilbert series of negative solutions to
the same equations.  In \cite[Theorem~4.2]{stanley1982linear} there
are, however, combinatorially defined correction terms.
\end{rem}

%

\begin{rem}
While this is not visible from the representation
in~\cite{KMplethysm}, a general theorem of Meinrenken and Sjamaar
\cite{meinrenken1999singular} implies that the chambers of plethysm
quasi-polynomials are polyhedral cones (see
\cite[Remark~3.12]{KMplethysm}, \cite{paradan2015witten}).  One may
thus ask the stronger question if they are Ehrhart quasi-polynomials
in general.  Of course, Theorem~\ref{t:main} gives a negative answer
to this much stronger property too.
\end{rem}

Lattice point counting in polytopes and Ehrhart functions are arguably
the most natural combinatorial explanations one may hope for here, but
they are not the only ones. Our theorem also excludes other
possibilities.
\begin{cor}
The plethysm coefficients are not positive combinations of counting
functions of lattice points in integral scalings of any convex
regions.
\end{cor}
\begin{proof}
One-dimensional convex sets are polyhedra and the counting function in
Theorem~\ref{t:main} would need to specialize to a positive
combination of Ehrhart functions.
\end{proof}
%
The plethysm counting function in Theorem~\ref{t:main} cannot be
written as a positive combination of Ehrhart quasi-polynomials as
these also have to satisfy Ehrhart reciprocity.  However, it can be
written as a sum of an asymptotic Ehrhart quasi-polynomial and an
honest Ehrhart quasi-polynomial as in the following proposition (which
is easy to confirm).
\begin{prop}\label{p:sum}
Fix a small rational $\varepsilon > 0$. Let
$P(s) = [\varepsilon,s\cdot\frac{1}{3}+\varepsilon]$ and
$Q(s) = s\cdot \{\frac{1}{2}\}$.  Then
$\phi(s) = \#(P(s)\cap\Z) + \#(Q(s) \cap \Z)$.
\end{prop}

\begin{rem}
An interesting open question is the following. Suppose we are given a
piecewise quasi-polynomial whose chambers are cones.  Is there a
finite computational test if the quasi-polynomial is an (asymptotic)
Ehrhart function?  Assume it is not an Ehrhart function globally, but
is an Ehrhart function on each individual ray.  Is there a finite
number of rays such that finding polytopes for each of the finitely
many rays proves the Ehrhart nature for all rays?  Is it some
multi-dimensional generalization of Ap\'ery sets known from numerical
semigroups?
\end{rem}

\begin{rem}
According to \cite[Remark~3.8]{KMplethysm} the function $\phi$ in
Theorem~\ref{t:main} also counts, for example, the multiplicity of
$\lambda = s\cdot (7+2t,5+2t,2t)$ in $S^3(S^{s\cdot (5+2t)})$ for any
positive $t$.  Therefore counterexamples also exist for strictly
interior rays.
\end{rem}

Recently, there has been a lot of interest in \emph{Kronecker
coefficients}~\cite{manivel2014asymptotics2,
briand2009reduced,sam2015stembridge,baldoni2015multiplicity}.  These
are different from plethysm, but also hard to compute and important in
geometric complexity theory.  It is known that on rays they are given
by quasi-polynomials \cite[Theorem 1]{manivel2014asymptotics} that do
not have to be Ehrhart functions~\cite{King}.  However, the problem is
open for specific rays that are important in GCT.  Emmanuel Briand
pointed us towards following observation.  Two row partitions are
useful because Hermite reciprocity applies to them.  Further, by the
work of Vallejo \cite{Vallejo}, cf.~\cite{macdonald1998symmetric}, the
plethysm coefficients in Theorem~\ref{t:main} are equal to special
Kronecker coefficients:
\[
\phi(s)=K(3^{4s},3^{3s},(7s,5s))=K((4s,4s,4s),(4s,4s,4s),(7s,5s)),
\]
where $K$ depends on three partitions and is the Kronecker
coefficient.  Indeed, let $p_n(a,b)$ denote the number of Young
diagrams with $n$ boxes contained inside an $a\times b$ rectangle.
Both the multiplicity of $\lambda$ with two rows in the plethysm
$S^a(S^b)$ and the Kronecker coefficient $K(a^{b},a^{b},\lambda)$ can
be expressed as the difference
\[
p_{\lambda_2}(a,b)-p_{\lambda_2-1}(a,b).
\]
For example by transposition, this immediately implies Hermite
reciprocity.
\begin{exm}
The multiplicity of $(7,5)$ inside $S^4(S^3)$ equals zero. Here we
have $4$ diagrams with $5$ boxes inside a $4\times 3$ rectangle:
$(4,1),(3,2),(3,1,1),(2,2,1)$. We also have $4$ such diagrams with $4$
boxes: $(4,0),(3,1),(2,2),(2,1,1)$, hence the difference is indeed
zero.
\end{exm}
The rectangular Kronecker coefficients $K(a^b,a^b,\lambda)$ are of
particular interest in GCT (see e.g.~\cite{ChristianandGretta}).
Moreover, they exactly govern unimodality of $q$-binomial
coefficients, observed already in 1878 by Sylvester (see
\cite[p.~2]{PakPanova}).

In a different direction, the following interesting question was
recently brought to our attention by Mich\`ele Vergne.
\begin{ques}
Are the plethysm coefficients Ehrhart polynomials after rescaling?
Specifically, are $\tilde f(s):=f(ks)$ or $\tilde g(s):=g(ks)$ Ehrhart
polynomials for some $k$?
\end{ques}
One can prove that the answer is positive when $f$ or $g$ are of
degree one, thus to find a counterexample one needs to consider
partitions with more than one row.

\subsection*{Acknowledgment}
The project was started at the Simons Institute at UC Berkeley during
the ``Algorithms and Complexity in Algebraic Geometry'' semester.
Micha{\l}ek was realizing the project as a part of the PRIME DAAD
program.  Micha{\l}ek would like to thank Klaus Altmann and Bernd
Sturmfels for the great working environment.  Kahle was supported by
the research focus dynamical systems of the state Saxony Anhalt.

\bibliographystyle{amsalpha}
\bibliography{plethysm}

\end{document}